\def\marker{\>\hbox{${\vcenter{\vbox{
    \hrule height 0.4pt\hbox{\vrule width 0.4pt height 6pt
    \kern6pt\vrule width 0.4pt}\hrule height 0.4pt}}}$}\>}

\documentclass[12pt]{article}

\usepackage[left=2cm,right=2cm,top=3cm,bottom=3cm]{geometry}

\usepackage{amsmath, amssymb, latexsym, amsthm}
\usepackage{fullpage}
\usepackage{mathrsfs, eufrak}
\usepackage[hang,flushmargin]{footmisc}

\usepackage{tikz}
\usetikzlibrary{shapes}

\newtheorem{theorem}{Theorem} 
\newtheorem{theorem*}{Theorem} 
\newtheorem{proposition}[theorem]{Proposition} 

\newtheorem{corollary}[theorem]{Corollary}
\newtheorem{lemma}{Lemma}
\newtheorem{conjecture}{Conjecture}

\theoremstyle{definition}

\newtheorem{observation}{Observation}

\theoremstyle{remark}
\newtheorem{claim}{Claim}

{\end{enumerate}}

\newcommand{\abs}[1]{\left\lvert{#1}\right\rvert}

\newcommand{\fl}[1]{\lfloor #1 \rfloor}

\newcommand{\f}{\mathcal{F}}
\newcommand{\C}{\mathcal{C}}
\newcommand{\R}{\mathcal{R}}

\newcommand{\M}{\mathcal{M}}

\DeclareMathOperator{\sat}{sat}

\makeatletter
\let\@fnsymbol\@arabic
\makeatother

\title{On Edge-Colored Saturation Problems}
\author{
Michael Ferrara \thanks{Department of Mathematical and Statistical Sciences, University of Colorado Denver, Denver, CO. \texttt{\{michael.ferrara,florian.pfender,eric.2.sullivan\}@ucdenver.edu}}
\and
Daniel Johnston \thanks{Department of Mathematics, Grand Valley State University, Allendale, MI. \texttt{johndan@gvsu.edu}}
\and
Sarah Loeb \thanks{Department of Mathematics, College of William and Mary, Williamsburg, VA. \texttt{sjloeb@wm.edu}}
\and
Florian Pfender \footnotemark[1]
\and
Alex Schulte \thanks{Department of Mathematics, Iowa State University, Ames, IA. \texttt{aschulte@iastate.edu}}
\and
Heather C. Smith \thanks{School of Mathematics, Georgia Institute of Technology, Atlanta, GA. \texttt{heather.smith@math.gatech.edu}} 
\and
Eric Sullivan \footnotemark[1]
\and
Michael Tait \thanks{Department of Mathematical Sciences, Carnegie Mellon University, Pittsburgh, PA. \texttt{mtait@cmu.edu}}
\and
Casey Tompkins \thanks{Alfr\'ed R\'enyi Institute of Mathematics, Hungarian Academy of Sciences. \texttt{ctompkins496@gmail.com}}
}

\begin{document}

\maketitle

\begin{abstract}
Let $\mathcal{C}$ be a family of edge-colored graphs. A $t$-edge colored graph $G$ is $(\mathcal{C}, t)$-saturated if $G$ does not contain any graph in $\mathcal{C}$ but the addition of any edge in any color in $[t]$ creates a copy of some graph in $\mathcal{C}$. Similarly to classical saturation functions, define $\mathrm{sat}_t(n, \mathcal{C})$ to be the minimum number of edges in a $(\mathcal{C},t)$ saturated graph. Let $\mathcal{C}_r(H)$ be the family consisting of every edge-colored copy of $H$ which uses exactly $r$ colors.

In this paper we consider a variety of colored saturation problems. We determine the order of magnitude for $\mathrm{sat}_t(n, \mathcal{C}_r(K_k))$ for all $r$, showing a sharp change in behavior when $r\geq \binom{k-1}{2}+2$. A particular case of this theorem proves a conjecture of Barrus, Ferrara, Vandenbussche, and Wenger. We determine $\mathrm{sat}_t(n, \mathcal{C}_2(K_3))$ exactly and determine the extremal graphs. Additionally, we document some interesting irregularities in the colored saturation function.

{\bf Keywords: saturation; edge-coloring} 
\end{abstract}


\renewcommand{\thefootnote}{\arabic{footnote}}

\baselineskip18pt

\section{Edge-colored Saturation Problems}

Given a family of graphs ${\mathcal F}$, a graph $G$ is \textit{${\mathcal F}$-saturated} if no $F\in {\mathcal F}$ is a subgraph of $G$, but for any $e\in E(\overline{G})$, some $F\in{\mathcal F}$ is a subgraph of $G+e$.  The minimum number of edges in an $n$-vertex ${\mathcal F}$-saturated graph is the {\it saturation number} of ${\mathcal F}$ and is denoted $\sat(n,{\mathcal F})$.  If ${\mathcal F}=\{F\}$, then we instead say that $G$ is {\it $F$-saturated}, and write $\sat(n,F)$.  
The saturation function was introduced by Erd\H os, Hajnal, and Moon \cite{EHM} and graph and hypergraph saturation problems have received considerable attention since that time.  We refer the interested reader to the dynamic survey of Faudree, Faudree, and Schmitt~\cite{FFS}, which contains a number of results and open problems.

In this paper, we are interested in saturation problems in the setting of edge-colored graphs, which was first introduced by Hanson and Toft \cite{HT} in 1987.  Let $[t]=\{1,2,\dots,t\}$.  A function $f: E(G) \rightarrow [t]$ is a $t$-edge-coloring of a graph $G$. A injective function yields a rainbow edge coloring.
Given a family $\C$ of edge-colored graphs, we say that a $t$-edge-colored graph $G$ is {\it $(\C,t)$-saturated} if $G$ contains no member of $\C$ as a (colored) subgraph, but for any  edge $e\in E(\overline{G})$ and any color $i\in [t]$, the addition of $e$ to $G$ in color $i$ creates some member of $\C$.
In line with classical saturation functions, we are interested in $\sat_t(n,\C)$, the minimum number of edges in a $(\C,t)$-saturated graph of order $n$. 

In this paper, we will primarily be interested in families of edge-colored graphs.  Given a graph $H$ and a fixed palette of $t$ colors, define $\M(H)$ to be the family of monochromatic edge colorings of $H$, $\R(H)$ denote the family of rainbow edge-colorings of $H$, and $\C_k(H)$ denote the set of edge colorings of $H$ using exactly $k$ of the $t$ colors.  Going forward, when considering $\sat_t(n,\C)$ for any of these functions, we will implicitly assume that we color these families from $[t]$.   Hanson and Toft \cite{HT} determined $\sat_t(n,\f)$ where $\f$ consists of monochromatic copies of $K_{t_i}$ in color $i$ for $1\le i\le t$, and also introduced a related conjecture that we will discuss briefly in the conclusion.

\subsection{Rainbow Subgraphs}  Barrus, Ferrara, Vandenbussche and Wenger~\cite{BFVW} introduced $\sat_t(n,\R(H))$ and considered several problems with a significant focus on the asymptotic behavior of this parameter for different choices of $H$.  As discussed in \cite{BFVW}, it is straightforward to show that for any graph $H$, $\sat_t(n,\M(H))=O(n)$.  This is not the case, however, for rainbow target graphs.  For instance, Barrus \textit{et al.} show that $\sat_t(n,\R(K_{1,r}))=\Theta(n^2)$ for $r\ge 2$, and gave two more general results that imply $$c_1 \frac{n \log(n)}{\log\log(n)} \le \sat_t(n,\R(K_k)) \le c_2 n\log(n)$$ for $k\ge 3$. They also conjectured the following.  

\begin{conjecture}\label{conj:rainbow_clique}
For $k\ge 3$ and $t\ge{k\choose 2}$, $$\sat_t(n,\R(K_k))=\Theta(n\log n).$$ 
\end{conjecture}

In Section \ref{sec:asymptotics}, we prove some broader results about $\sat_t(n,\C_k(H))$.
As a consequence of our result there, we prove the following.  

\begin{theorem}\label{thm:clique_asymp}
Let $k\ge 3$, and $t\ge c$ be fixed.  

\begin{enumerate}
\item If $c\ge {k-1\choose 2}+2$, then $\sat_t(n,\C_c(K_k))=\Theta(n\log n).$

\item If $c\le {k-1\choose 2}+1$, then $\sat_t(n,\C_c(K_k))=\Theta(n).$
\end{enumerate}
\end{theorem}

Independent of our work here, Gir\~{a}o, Lewis and Popielarz \cite{GLP} determined the asymptotics of $\sat_t(n,\R(H))$ for every connected graph $H$ without a pendant edge, and as a consequence also resolve Conjecture \ref{conj:rainbow_clique} in the affirmative.  Furthermore, Kor\'{a}ndi \cite{K} recently showed that $$\sat_t(n,\R(K_k))=\Theta_k \left(\frac{n\log n}{\log t} \right)$$ and gave sharp asymptotics (in $t$ and $n$) for $\sat_t(n,\R(K_3)).$ The techniques utilized across all three papers are quite diverse, and provide an interesting spectrum of possible approaches to problems of this type.  

\subsection{Irregularities}

It has been well-documented (see, for instance \cite{HT,FFS}) that the classical (uncolored) saturation function is not monotone in $n$ or with respect to subgraph and family inclusion.  That is, there is a graph $H$ such that $\sat(n,H)\le\sat(n+1,H)$ for infinitely many $n$, distinct graphs $H_1\subseteq H_2$ such that $\sat(n,H_2)\le \sat(n,H_1)$, and distinct families $\f_1\subseteq\f_2$ such that $\sat(n,\f_2)\le\sat(n,\f_1)$.  Before continuing on to our main results, we give some examples of similar irregular behaviors for $\sat_t(n,\C_k(H))$.  

Recall that it was shown in \cite{BFVW} that $\sat_t(n,\R(K_{1,k}))=\Theta(n^2)$ for all $t\geq k \ge 3$, and further that if $T$ is any tree with $k\geq 4$ vertices that is not a star, then $\sat_t(n,\R(T))=O(n)$ when $t\geq \binom{k-1}{2}$.  This immediately establishes that $H_1\subseteq H_2$ does not necessarily imply that $\sat_t(n,\R(H_1))\le \sat_t(n,\R(H_2))$.  More interestingly, in our opinion, is the following result, which establishes that $\sat_t(n,\C_k(H))$ is not monotone (increasing or decreasing) in $k$.  Recall that $\M(H) = \C_1(H)$ and $\R(H)=\C_{|E(H)|}(H).$

\begin{theorem}\label{prop:non-mon}
For $t\ge 3$ and $n$ sufficiently large, $$\sat_t(n,\C_2(K_{1,3}))<\sat_t(n,\C_1(K_{1,3}))<\sat_t(n,\C_3(K_{1,3})).$$
\end{theorem}

First recall $\sat_t(n,\C_1(K_{1,3}))=O(n)$ and $\sat_t(n,\C_3(K_{1,3}))=\Theta(n^2)$~\cite{BFVW}. The theorem then follows from two propositions.  

\begin{proposition}\label{C_1(K_{1,3})} For $n \ge 2t\geq 6$, $\sat_t(n,\C_{1}(K_{1,3})) \ge tn/2$. 
\end{proposition}
\begin{proof}
Suppose $G$ is saturated for $\C_{1}(K_{1,3})$. If we have a non-edge $uv$ in $G$, then adding $uv$ to $G$ in any color produces a monochromatic $K_{1,3}$ at $u$ or at $v$. Since $v$ can be saturated for at most $\fl{d(v)/2}$ colors, there are at least $t-\fl{d(v)/2}$ colors in which we could add edge $uv$ without producing a monochromatic $K_{1,3}$ at $v$.  If $d(u) < 2(t-\fl{d(v)/2})$, then there is some color in which $uv$ can be added that produces neither a monochromatic $K_{1,3}$ at $u$ or at $v$. Thus if $uv \notin E(G)$, then $d(u) \ge 2(t-\fl{d(v)/2})$. 

If $\delta(G) \ge t$, then we get that $e(G) \ge tn/2$ so we may assume that $\delta(G) = \ell < t$. In this case, rather than directly compute the number of edges in $G$, we consider the degree sum. We will condition on whether or not a vertex is in the closed neighborhood of a fixed vertex of minimum degree. To this end, let $v$ be a vertex with $d(v) = \ell$. For $u \notin N[v]$, we have $d(u) \ge 2(t-\fl{\ell/2})$ and, for all vertices, we have $d(u) \ge \ell$.
\begin{align*}
\sum_{w \in V(G)} d(w) &= \sum_{w \notin N[v]} d(w) + \sum_{w \in N[v]} d(w) \\
&\ge (n-(\ell+1))\cdot (2t-\ell) + (\ell+1)\cdot \ell \\
&= n(2t-\ell) - 2(\ell+1)(t-\ell)\\
&=nt + (t-\ell)[n-2(\ell+1)]
\end{align*}
For $\ell \le t-1$, 
the degree sum is at least $nt+(n-2t)$. 
Thus when $n \ge 2t$, we find $e(G) \ge nt/2$. 
\end{proof}

Next, we give an upper bound on $\sat_t(n,\C_2(K_{1,3}))$ by providing a more general saturated graph for $\C_{s-1}(K_{1,s})$. This suffices to complete the proof of Theorem \ref{prop:non-mon}.

\begin{proposition}\label{C_s(K_{1,s-1})} For $n$ sufficiently large with respect to $t$ and $t \ge s-1$, $\sat_t(n,\C_{s-1}(K_{1,s})) < tn/2$. 
\end{proposition}

\begin{proof}
Let $n = 2k+r$ where $r \in \{1,2\}$. Let $H$ be a graph on $2k$ vertices produced by packing $t$ perfect matchings onto $V(H)$, with each matching in a distinct color. Let $G$ be the graph that is the disjoint union of $H$ and $K_r$, where edges in the $K_r$ are colored arbitrarily. Then $G$ is a $\C_{s-1}(K_{1,s})$-saturated graph; any added edge has at least one endpoint in $H$ and results in an $(s-1)$-colored $K_{1,s}$ centered at that vertex. 
\end{proof}

It would be interesting to determine if there exist other graphs exhibiting such unusual behavior.  For instance, for any pattern $x_1,\dots,x_{k-1}$ of $``\uparrow"$ and $``\downarrow"$ symbols, does there exist a graph $H$ with $k$ edges such that $\sat_t(n,\C_i(H))\ge \sat_t(n,\C_{i+1}(H))$ if and only if $x_i = \uparrow$?

\section{Asymptotics}\label{sec:asymptotics}

In this section, we prove the following general result that implies Conjecture \ref{conj:rainbow_clique}.  Together with Theorem \ref{thm:linear_clique}, this result implies Theorem \ref{thm:clique_asymp}.

\begin{theorem}\label{thm:2-colored-path}
Let $\mathcal{H}$ be a family of edge-colored graphs where for each $H\in \mathcal{H}$, for each edge $uv\in E(H)$ there is a rainbow path with $2$ edges connecting $u$ to $v$ in $H$. Then for any integer $t$, $t\geq 3$, we have 
\[
\left(\frac{1}{3} - o(1)\right)\frac{n\log n}{\log t} \leq \sat_t(n,\mathcal{H}).
\]
\end{theorem}

Before we proceed with the proof of this theorem, we make the following simple observation that will be useful going forward.

\begin{observation}\label{observation rainbow}
If $\mathcal{H}$ is as in Theorem \ref{thm:2-colored-path} and if $G$ is $(\mathcal{H},t)$-saturated, then for any nonedge $uv$ in $G$ there is a 2-edge path with two colors connecting $u$ to $v$ in $G$.
\end{observation}

We use this observation to prove Theorem \ref{thm:2-colored-path} via a reduction to a specific covering problem.
Let $\mathcal{F}$ be a family of complete $t$-partite graphs with $U^F_1, \cdots, U^F_t$ the partite sets of $F\in \mathcal{F}$. We say that $\mathcal{F}$ is a $t$-partite cover of a graph $H$ if $E(H) \subseteq \bigcup_{F\in \mathcal{F}} E(F)$. Define
\[
f(H) = \min_{\mathcal{F}} \sum_{F \in \mathcal{F}} |U^F_1|+ \cdots + |U^F_t|,
\]
where the minimum is taken over all $\mathcal{F}$ a $t$-partite cover of $H$.

Let $\mathcal{H}$ be a family of $t$-edge-colored graphs where for each $H\in \mathcal{H}$ and each edge $uv\in E(H)$ there is a 2-edge path with two colors connecting $u$ to $v$ in $H$. Assume $G$ is $(\mathcal{H},t)$-saturated. We create a $t$-partite cover of the complement of $G$. For each vertex $v$ and $1\leq i\leq t$, let $\Gamma_i(v)$ be the set of vertices adjacent to $v$ in $G$ with edge color $i$. For each vertex $v$, let $G_v$ be the complete $t$-partite graph on $V(G)$ with partite sets $\Gamma_1(v),\cdots, \Gamma_t(v)$. By Observation \ref{observation rainbow}, if $x$ and $y$ are not adjacent in $G$, then there is a rainbow path of length $2$ between them. If the vertex in the middle of this path is $v$, then $G_v$ contains the edge $xy$. Therefore, $\bigcup_{v\in V(G)} G_v$ is a $t$-partite cover of the complement of $G$. 

Next we note that 
\[
\sum_v |\Gamma_1(v)| + \cdots + |\Gamma_t(v)| = \sum_v d(v) = 2e(G).
\]
Therefore, we have that if $G$ is $(\mathcal{H},t)$-saturated, then 
\begin{equation}\label{lower bound by cover}
f(\overline{G}) \leq 2e(G).
\end{equation}

We need the following lemma, which we modify from a result of Katona and Szem\'eredi \cite{KS}.

\begin{lemma}\label{covering complete graphs}
\[
f(K_n) \geq \frac{n\log n}{\log t}.
\]

\end{lemma}

\begin{proof}
Let $\mathcal{F} = \{F_j\}_{j=1}^\ell$ be a $t$-partite cover of $K_n$ where $F_j$ has partite sets $U_1^j,\cdots, U_t^j$. Create a matrix $M$ with the rows indexed by $V(K_n)$ and the columns indexed by $\mathcal{F}$ as follows:
\[
M_{ij} = \begin{cases}
k & \mbox{if $i\in U_k^j$,}\\
* & \mbox{if $i$ is not in any partite set of $F_j$.}
\end{cases}
\]
For each vertex $v$, let $d_v$ be the number of entries which are not $*$ in the row corresponding to $v$ (i.e. $d_v$ is the number of $t$-partite graphs in $\mathcal{F}$ which use $v$). Then in the row corresponding to $v$ there are $|\mathcal{F} |- d_v$ entries with $*$. We create a new matrix $M'$ where for each row $v$ we replace it with $t^{|\mathcal{F}| - d_v}$ rows putting all possible replacements of $*$ with elements from $\{1,\ldots, t\}$ and leaving all other entries the same.

We claim that each row in $M'$ is distinct. To see this, if a pair of rows are in the $t^{|\mathcal{F}| - d_v}$ rows which correspond to the same vertex $v$, then the replacements of $*$ with $\{1,\ldots, t\}$ will be different in at least one position. If a pair of rows in $M'$ correspond to distinct vertices $u$ and $v$, then because $\mathcal{F}$ is a $t$-partite cover of $K_n$, there is an $F\in \mathcal{F}$ where $u$ and $v$ are in different partite sets of $F$ and therefore there is a column in $M'$ that distinguishes the two rows. 

Since the total number of distinct rows is at most $t^{|\mathcal{F}|}$, we have 
\[
\sum_{v} t^{|\mathcal{F}| - d_v} \leq t^{|\mathcal{F}|}
\]
and therefore
\[
\sum_v \frac{1}{t^{d_v}} \leq 1.
\]

Now the AM-GM inequality implies 
\[
\sqrt[n]{\prod_v \frac{1}{t^{d_v}}} \leq \frac{1}{n} \sum_v \frac{1}{t^{d_v}} \leq \frac{1}{n}.
\]
Rearranging gives $\sum_v d_v \geq n\log_t n$. Noting that $\sum_v d_v = \sum_{j=1}^{\ell} |U_1^j| + \cdots |U_t^j|$ finishes the proof.
\end{proof}

Next we need to show that covering $K_n$ with $t$-partite graphs is not much different from covering the complement of a sparse graph with $t$-partite graphs.

\begin{lemma}\label{covering incomplete graphs}
Let $H$ be a graph on $n$ vertices. Then 
\[
f(\overline{H}) \geq f(K_n) - (e(H) + n).
\]
\end{lemma}

\begin{proof}
Let $\mathcal{F}$ be an $t$-partite covering of $\overline{H}$ with weight $f(\overline{H})$. We will construct a family of $t$-partite graphs that covers $H$ with weight at most $e(H) + n$, which will certify that 
\[
f(K_n) \leq f(\overline{H}) + (e(H) + n).
\]

Order the vertices of $H$ arbitrarily as $v_1,\cdots, v_n$. For $1\leq i\leq n$ define a $t$-partite graph $H_i$ with partite sets $U^i_1, \cdots, U^i_t$ where 
\begin{align*}
U^i_1 &= \{v_i\},\\
U^i_2 & = \{v_j :  v_i\sim v_j, j>i\},\\
U^i_k & = \emptyset \quad \mbox{(for $3\leq k\leq t$)}.
\end{align*}
Then $\{H_i\}_{i=1}^n$ partitions the edge set of $H$ into stars, and 
\[\sum_{i} |U^i_1| + \cdots + |U^i_t| = e(H) + n.\]
\end{proof}

We are now ready to complete the proof of Theorem \ref{thm:2-colored-path}

\begin{proof}[Proof of Theorem \ref{thm:2-colored-path}]
Let $G$ be $(\mathcal{H},t)$-saturated and assume for a contradiction that $e(G) < \frac{n\log n}{3\log t} - n/2$. Then
$\bigcup_{v\in V(G)} G_v.$
is a $t$-partite cover of $\overline{G}$, implying that $f(\overline{G}) \leq 2e(G) < \frac{2n\log n}{3\log t} - n$. This, together with Lemma \ref{covering incomplete graphs} implies $f(K_n) < \frac{n\log n}{\log t}$, which contradicts Lemma \ref{covering complete graphs}.
\end{proof}

The following corollary to Theorem 4 implies Conjecture \ref{conj:rainbow_clique}, and resolves Conjecture 2 in \cite{GLP}.  This conjecture was also affirmed in \cite{K}, where the focus was strictly the rainbow setting.

\begin{corollary}\label{cor:nlogn}
If $c\ge{k-1\choose 2}+2$ and $t\ge c$, then $$\sat_t(n,\C_c(K_k))\ge \left(\frac{1}{3}-o(1)\right)\frac{n\log n}{\log t}.$$
\end{corollary}

\begin{proof}
Let $c\ge{k-1\choose 2}+2$ and consider an edge $uv$ in $G$, an edge-colored $K_k$ with exactly $c$ colors. There are at most ${k-2\choose 2}$ colors on the edges in $G-\{u,v\}$, and at most one additional color on $uv$.  This leaves at least $c-{k-2\choose 2}\ge k-1$ colors on the edges with one endpoint in $\{u,v\}$ and one endpoint in $V(G)-\{u,v\}$, implying that there is some vertex $x$ such that the edges of $uxv$ receive distinct colors.  We can therefore apply Theorem \ref{thm:2-colored-path} to $\sat_t(n,\C_c(K_k)).$
\end{proof}

As we demonstrate next, the bound of $c\ge {k-1\choose 2}+2$ is sharp.  

\begin{theorem}\label{thm:linear_clique}
If $c\le {k-1\choose 2}+1$ and $t\ge c$ are fixed, then $\sat_t(n,\C_c(K_k)) = O(n).$ 
\end{theorem} 

\begin{proof}
For fixed $c\le {k-1\choose 2}+1$ and $t\ge c$, we construct a $\C_c(K_k)$-saturated graph with $O(n)$ edges.  As we are not interested in determining the relevant saturation number exactly, we make no effort to optimize the number of edges in our construction.

Assume that $n$ is sufficiently large, and consider an edge-coloring of $H'=K_{k-1}$ using exactly $c-1$ colors. Create an edge-colored copy of $K_k-e$ by choosing some vertex $v$ in $H'$, adding a new vertex $v'$, and connecting $v'$ to $V(H')-\{v\}$ such that $vx$ and $v'x$ have the same colors for each $x\in V(H')-\{v\}$.  Repeat the duplication of $v$ to create $H_{S,p}=K_{k-2}\vee I$, where $I$ is an independent set of size $p$ and $S$ is the set of colors appearing on $E(H_{S,p})$.  Note that $H_{S,p}$ contains no copy of $K_{k}$, but the addition of any edge to $H_{S,p}$ in a color from $[t]-S$ creates a copy of $K_k$ with exactly $c$ colors. 

We create the edge-colored graph $G'$ by taking the union of the graphs $H_{S,p}$ with $p=n-{t\choose c-1}(k-2)$ for each of the ${t\choose c-1}$ choices of $S$, under the assumption that $I$ is common to each such graph. Note that for any $u$ and $v$ in $I$ and any color ${c_0}\in[t]$, adding $uv$ in color ${c_0}$ to $G'$ creates a copy of $K_k$ with exactly $c$ colors within $H_{S, p}$ for any $S$ that does not contain ${c_0}$.  To create the desired saturated graph $G$, iteratively add edges to $G'-I$ in any permissible color until either $G-I$ is complete, or no colored edge can be added to $G-I$ without creating an element of $\C_c(K_k)$.  In either case, $G$ is $\C_c(K_k)$-saturated and has at most $${t\choose c-1}(k-2)\left(n-{t\choose c-1}(k-2)\right)+{{t\choose c-1}(k-2)\choose 2}$$ edges, which is $O(n)$ edges as desired.  
\end{proof}

Theorem \ref{thm:clique_asymp} now follows as a consequence of Corollary \ref{cor:nlogn} and Theorem \ref{thm:linear_clique}.

\section{2-Colored Triangles}

In this section, we prove the following exact result.  

\begin{theorem}\label{thm:2-colored-triangles}
If $t = 2$ and $n\ge 11$ or if $t\ge 3$ and $n\ge 9$, then $$\sat_t(n,\C_2(K_3))=2n-4.$$ Furthermore, if $t\ge 3$, then $K_{2,(n-2)}$ is the unique saturated graph.
\end{theorem}

Before we proceed, we require a simple technical lemma.   

\begin{lemma}
\label{convexity}
Let $x_1, x_2,\dots,x_t$ be  integers with $x_1 \ge x_2 \ge \dots \ge x_t\geq 2$. For $1 \le p < q \le t$ and let $x_i' := x_i$ for $i \not\in \{p,q\}$, $x_p' := x_p+1$ and $x_q' = x_q-1$, then
\begin{displaymath}
\sum_{i=1}^t \binom{x_i}{2} < \sum_{i=1}^t \binom{x_i'}{2}.
\end{displaymath}
\end{lemma}
\begin{proof}
We must show that
\begin{displaymath}
\binom{x_p+1}{2} + \binom{x_q-1}{2} > \binom{x_p}{2}+\binom{x_q}{2},
\end{displaymath}
but this is equivalent to showing $x_p > x_q-1$, which holds by assumption.
\end{proof}

\begin{proposition}
\label{props}
For all $n \ge 11$,
\begin{displaymath}
\sat_2(n,\C_2(K_3))= 2n-4.
\end{displaymath}
\end{proposition}

\begin{proof}
Consider the edge-colored graph $K_{2,(n-2)}$ where $x$ and $y$ are the vertices in the partite class of size 2 and all edges incident with $x$ are red while all edges incident with $y$ are blue. This shows that $\sat_2(n,\C_2(K_3)) \le 2n-4$ (other constructions exist).  We will show $\sat_2(n,\C_2(K_3)) \ge 2n-4$.

Suppose $G$ is an $n$-vertex $(\C_2(K_3),2)$-saturated graph.  We will consider cases based on the minimum degree of $G$.

\textit{Case 1:} $\delta(G)=1$. \\
Let $u \in V(G)$ be of degree 1 with neighbor $v$.  Say that the color of $uv$ is blue.  Then, $v$ must be adjacent to every other vertex $w$ in $G$ by a red edge for otherwise we could add the blue edge $uw$.  There is no blue edge in $G - \{u,v\}$ for this would yield a $\C_2(K_3)$. Thus, $G- \{u,v\}$ is a complete graph consisting of only red edges. Thus $e(G) \ge 2n-3$.

\textit{Case 2:} $\delta(G)=2$.\\
Let $u$ be a vertex with degree $2$ and suppose $v_1$ and $v_2$ are the neighbors of $u$.  If there is a vertex $y$ which is not adjacent to $u$, $v_1$, or $v_2$, then the graph would not be saturated since we could add the edge $uy$ in either color and not obtain a triangle. So every vertex is adjacent to $u$, $v_1$, or $v_2$.  

 Suppose first that $v_1v_2$ is an edge in $G$.  Then, $v_1,v_2$ and $u$ form a triangle and so all 3 edges are the same color, say blue. The only common neighbor of $v_1$ and $v_2$ is $u$ since any further neighbor $w$ would be connected to $v_1$ and $v_2$ by blue edges. But then the graph would not be saturated since we could add the edge $uw$ in blue. Now $X:=V(G)- \{u,v_1,v_2\}$ is the set of vertex which are neighbors of $v_1$ or $v_2$. Note that every edge from $\{v_1,v_2\}$ to a vertex $x \in X$ is red, for otherwise we could add the blue edge $xu$.  The graph induced on the vertex set $X$ is connected since two connected components could be connected by a red edge without forming a $\C_2(K_3)$.  Thus, we have at least $3 + \abs{X} + \abs{X}-1 = 2n-4$ edges.   From now on we may assume that $v_1v_2$ is not an edge.

First, suppose that both $uv_1$ and $uv_2$ are blue. Let $X$ denote the set of vertices in $V(G) - \{u,v_1,v_2\}$ which are connected to exactly one of $v_1$ and $v_2$, and let $Y$ denote the set of vertices in $V(G) - \{u,v_1,v_2\}$ which are connected to both $v_1$ and $v_2$. So $X\cup Y = V(G) - \{u,v_1,v_2\}$. All edges from $X$ to $\{v_1,v_2\}$ are red, and at least one edge from every vertex of $Y$ to $\{v_1,v_2\}$ is red. If $X = \varnothing$ we are done, so suppose there is at least one vertex in $X$.  The graph induced on the set $X$ must be connected for otherwise we could add a red edge connecting two components.  Thus, there are at least $\abs{X}-1$ edges in this induced subgraph, yielding a total of at least $2 + 2\abs{Y} + \abs{X} + \abs{X}-1 = 2n - 5$ edges.  We now suppose that these $2n-5$ edges are the only edges in $G$ and argue to a contradiction.  

 Let $N_1$ and $N_2$ denote the neighborhoods of $v_1$ and $v_2$ in $X$ respectively. Since there are no edges with one endpoint in $X$ and the other in $Y$, there must be a blue edge in the induced subgraph on $X$ for otherwise we could add a red edge from a vertex in $X$ to either $v_1$ or $v_2$, the one it is not connected to. This blue edge has one vertex in $N_1$ and one vertex in $N_2$, say $w_1$ and $w_2$, respectively. Observe that $Y$ must be nonempty for otherwise we could add the blue edge $v_1v_2$.  Let $z \in Y$ and assume $v_1z$ is red. Then we could add the edge $w_1z$ with color red. Thus,  the graph $G$ could not be saturated with $2n-5$ edges and so contains at least $2n-4$.

Second suppose $uv_1$ is blue and $uv_2$ is red. Define $X$ and $Y$ as in the previous subcase.  Again assume $X$ is nonempty for otherwise we have $2n-4$ edges. But now $Y$ may be empty.  Edges from $v_1$ to $X$ are red, and edges from $v_2$ to $X$ are blue. Again, the graph induced on $X$ is connected.  Thus, again we have at least $2n-5$ edges. Now suppose that there are no other edges in $G$ and we will argue to a contradiction. 
If $N_2$ was empty, then for $w\in N_1$, we could add the edge $wv_2$ and not create a triangle. Since $N_1, N_2 \neq \varnothing$ and the graph induced on $X$ is connected, there is an edge $w_1w_2$ with $w_1 \in N_1$ and $w_2 \in N_2$. Suppose $w_1w_2$ is red (the blue case is similar), then we can add the edge $v_1w_2$ with color red.  Thus $2n-5$ edges do not suffice for $G$. Thus $G$ must have at least $2n-4$ edges.

\textit{Case 3:} $\delta(G)=3$.\\
We need to show that $\sum_v{d(v)} \ge 4n-8$. Suppose by way of contradiction that $\sum_v{d(v)} \le 4n-10$.  For every edge which is not in $G$, there must exist a path of length 2 between its endpoints.    It follows that
\begin{equation}
\label{constraint1}
\sum_{v \in V(G)} \binom{d(v)}{2} \ge \binom{n}{2} - \frac{1}{2}\sum_{v \in V(G)}{d(v)} \ge  \binom{n}{2} - \frac{1}{2}(4n-10) = \frac{n^2}{2}- \frac{5n}{2} +5.
\end{equation}
On the other hand, by Lemma~\ref{convexity} we have
\begin{equation}
\label{constraint2}
\sum_{v \in V(G)} \binom{d(v)}{2} \le \binom{n-7}{2} + (n-1)\binom{3}{2} = \frac{n^2}{2} - \frac{9n}{2} + 25.
\end{equation} 
It follows from \eqref{constraint1} and \eqref{constraint2} that $n \le 10$, a contradiction.
\end{proof}

\begin{proposition}
\label{propt}
For all $t \ge 3$ and $n \ge 9$,
\begin{displaymath}
\sat_t(n,\C_2(K_3)) = 2n-4.
\end{displaymath}
Moreover, every $(\C_2(K_3),t)$-saturated graph is a coloring of $K_{2,n-2}$.
\end{proposition}

\begin{proof}[Proof of Proposition \ref{propt}]
A construction is given by the following. Take two vertices $u$ and $v$ and a collection of vertices $u_1,u_2,\dots,u_{n-2}$.  Take red edges from $u$ to $u_1,\dots,u_{n-2}$ and blue edges from $v$ to $u_1,\dots,u_{n-3}$ and a red edge from $v$ to $u_{n-2}$.  Therefore $\sat_t(n,\C_2(K_3)) \leq 2n-4.$

Now we will establish the lower bound. Let $G$ be an $n$-vertex $(\C_2(K_3),t)$-saturated graph with $t \ge 3$ with $e(G)$ as small as possible.   First, we show that the minimum degree of $G$ is at least 2.   Suppose $u\in V(G)$ is a vertex of degree 1 with neighbor $v$ and let $w$ be any other vertex.  If $vw$ is not an edge, then we can add $uw$ without creating a triangle.  If $vw$ is an edge with the same color as $uv$, then we may add $uw$ with the same color.  If $vw$ is an edge with a different color than $uv$, then we may add $vw$ with an arbitrary distinct third color (since $t \ge 3$). 

We need to show that $\sum_{v\in V(G)} d(v) \ge 4n-8$.  Observe that if $G$ is  $(\C_2(K_3),t)$-saturated, then for every edge $e=\{x,y\} \in E(\overline{G})$ there must be at least two paths of length 2 between $x$ and $y$.   Since the number of paths of length 2 in $G$ is $\sum_{v \in V(G)} \binom{d(v)}{2}$, it follows that

\begin{equation}
\label{main} 
\sum_{v \in V(G)} \binom{d(v)}{2} \ge 2 \left(\binom{n}{2} - e(G)\right) = n^2-n - \sum_{v\in V(G)} d(v).
\end{equation}

Under the assumption that $\sum_{v\in V(G)} d(v) \le 4n-10$, the right hand side of  \eqref{main} is at least $n^2-5n+10$. By Lemma \ref{convexity}, if $\sum_{v\in V(G)} d(v) \le 4n-10$, then 
\begin{displaymath}
\sum_{v \in V(G)}\binom{d(v)}{2} \le \binom{n-1}{2} + \binom{n-5}{2} + (n-2)\binom{2}{2} = n^2-6n+14.
\end{displaymath}
This is a contradiction for $n \ge 5$.

Since $\sum_{v\in V(G)} d(v)$ cannot be odd, it remains to show that if $G$ is $(\C_2(K_3),t)$-saturated with $\sum_{v\in V(G)} d(v) = 4n-8$, then $G$ is a coloring of $K_{2,n-2}$. In this case the right hand side of \eqref{main} is $n^2-5n+8$.  

 First, we observe that it is not possible for the maximum degree to be at most $n-3$, for then (by Lemma \ref{convexity}) we would have
 \begin{displaymath}
\sum_{v \in V(G)} \binom{d(v)}{2} \le \binom{n-3}{2} + \binom{n-3}{2} + \binom{4}{2}+ (n-3)  = n^2-6n+15,
\end{displaymath}
which is too small for $n \ge 8$.  

Suppose the maximum degree of $G$ is $n-2$.  We see that the second largest degree is at least $n-3$ for otherwise we have 
 \begin{displaymath}
\sum_{v \in V(G)} \binom{d(v)}{2} \le \binom{n-2}{2} + \binom{n-4}{2} + \binom{4}{2} + (n-3)  = n^2-6n+16,
\end{displaymath}
which is too small for $n \ge 9$.  Thus, the remaining possible degree sequences starting with $n-2$ are $(n-2,n-2,2,\dots,2)$ and $(n-2,n-3,3,2,\dots,2)$. Note that $(n-2,n-2,2,\dots,2)$ yields a colored $K_{2,n-2}$.

Suppose the maximum degree is $n-1$.  If the second largest degree is at most $n-5$, then we have
 \begin{displaymath}
\sum_{v \in V(G)} \binom{d(v)}{2} \le \binom{n-1}{2} + \binom{n-5}{2} + \binom{4}{2} + (n-3)  = n^2 - 6n +19,
\end{displaymath}
which is too small $n \ge 9$. The remaining possible degree sequences are $(n-1,n-4,3,2,\dots,2)$ and $(n-1,n-3,2,\dots,2)$.  To finish the proof we have to check the 3 degree sequences which satisfied \eqref{main}.  We do this in Claim \ref{casecheck} below.
\end{proof}

\begin{claim}
\label{casecheck}
For $t \ge 3$, there is no $(\C_2(K_3),t)$-saturated graph with any of the following degree sequences:
\begin{itemize}
\item $(n-2,n-3,3,2,\dots,2)$,
\item $(n-1,n-3,2,\dots,2)$,
\item $(n-1,n-4,3,2,\dots,2)$.
\end{itemize}
\end{claim}
\begin{proof}
First, it is important to note that if $G$ is $(\C_2(K_3),t)$-saturated, then for any $uv\not\in E(G)$, there must be at least two paths of length 2 from $u$ to $v$ because $t\geq 3$. 

Consider the degree sequence $(n-2,n-3,3,2,\dots,2)$ and let $x$ and $y$ be the vertices of degree $n-2$ and $n-3$, respectively. If $x$ and $y$ are not adjacent let $z$ be the other vertex $y$ is not adjacent to.  Then $z$ has degree at least 2 so it is adjacent to some other vertex.  This defines a unique graph up to isomorphism and it is clear that for the nonedge $yz$ there is only one path of length 2 from $y$ to $z$. 

Now, assume  $x$ and $y$ are adjacent. If there exists a common vertex $z$ such that $xz$ and $yz$ are both nonedges, then the remainder of the graph is forced.  Namely, there is an edge from $z$ to the other non-neighbor of $y$ and an edge from $z$ to one of the common neighbors of $x$ and $y$.  For the nonedge $yz$ there are not two paths of length 2 between $y$ and $z$.    Finally, assume $x$ and $y$ are adjacent and the set of non-neighbors of $x$ and the set of non-neighbors of $y$ are disjoint.  In this case there are multiple non-isomorphic graphs but a nonedge from one of the non-neighbor sets to a vertex in the set of common neighbors of degree 2 will not have two paths of length 2 between its endpoints.  (Indeed, any such path of length 2 would involve $x$ or $y$ but one of the vertices in the nonedge is a non-neighbor of $x$ or $y$.)

Consider the degree sequence $(n-1,n-3,2,\dots,2)$.  Let $x$ and $y$ be the degree $n-1$ and $n-3$ vertex, respectively. It is easy to see the two non-neighbors of $y$ must be adjacent and this defines a unique graph.  Then a nonedge from $y$ does not have two paths between its endpoints.

Consider, finally, the degree sequence   $(n-1,n-4,3,2,\dots,2)$. Let $x$ and $y$ be the degree $n-1$ and $n-4$ vertex, respectively.  Let $u,v$ and $w$ be the three non-neighbors of $y$.  Either these three vertices form a path of length 2, or we may assume $u$ and $v$ are adjacent and $w$ is adjacent to a common neighbor of $x$ and $y$.  In either case the nonedge $yu$ does not have two paths of length 2 between its vertices.  \end{proof}

\section{Conclusion}

In this paper, we consider several existing and new problems in the realm of edge-colored saturation problems.  There remain a number of potential directions of inquiry.  Even given the excellent results in \cite{GLP} and \cite{K}, the general problem of determining $\sat_t(n,\R(H))$ is open in a number of cases.  Of particular interest would be to determine the asymptotic behavior for general trees, or to consider the behavior of the function for disconnected graphs.  For instance, it is not difficult to show that if $p$ is even, $n\ge 5p$, and $t$ is large, then $\sat_t(n,\R((p+1)K_2))\le 5p$.  The extremal graph is a rainbow copy of $\frac{p}{2}K_5$ together with $n-\frac{5p}{2}$ isolated vertices. However, it seems surprisingly difficult to show that equality holds.  

We also point out that the families considered here, $\M(H), \R(H)$ and $\C_k(H)$ are invariant up to the permutation of the palette of $t$ colors.  What if this was not the case?  Suppose, for instance, that we wished to determine $\sat_3(n,\f)$, where $\f$ consisted of two graphs:  a triangle with two edges colored 1 and one edge colored 2, and a monochromatic triangle with all edges colored 3.  In this case, not all colored edges are created equal, opening the door to a number of (delightfully) aberrant possibilities.  

Finally, in \cite{HT}, Hanson and Toft also introduced the related problem of determining the saturation number of the family of graphs that are Ramsey-minimal for some $(H_1,\dots,H_t)$.  This is equivalent to determining the minimum number of edges in a graph of order $n$ that has a $t$-edge-coloring with no copy of $H_i$ in color $i$, such that the addition of any missing edge creates a graph wherein every $t$-edge-coloring contains some $H_i$ in color $i$.  While not our focus here, we want to highlight this general problem, which has only been considered for a limited collection of target graphs \cite{CFGMS,FKY,KS} and remains open.  

\section{Acknowledgements}
All authors were supported in part by NSF-DMS Grants \#1604458, \#1604773, \#1604697 and \#1603823, ``Collaborative
Research: Rocky Mountain - Great Plains Graduate Research Workshops in Combinatorics.'' Pfender is supported in part by NSF-DMS grant \#1600483. Smith is supported in part by NSF-DMS grant \#1344199. Tait is supported by NSF-DMS grant \#1606350. Tompkins is supported by NKFIH grant \#K116769.  Ferrara is supported in part by Simons Foundation Collaboration Grant \#426971.

\end{document}